\documentclass[12pt]{amsart}
\usepackage{amsmath}
\usepackage{amssymb}
\usepackage{amscd}
\usepackage{multicol}
\usepackage{MnSymbol}
\usepackage{tikz}
\usepackage[OT2,T1]{fontenc}
\usepackage{mathrsfs}
\usepackage{comment} 
\usepackage{stmaryrd}
\usepackage[all]{xy}
\usepackage{longtable}
\usepackage{multirow,bigdelim}
\DeclareSymbolFont{cyrletters}{OT2}{wncyr}{m}{n}
\DeclareMathSymbol{\Sha}{\mathalpha}{cyrletters}{"58}
 
 \DeclareMathOperator{\GL}{GL}

 \DeclareMathOperator{\wt}{wt}

\DeclareMathOperator{\dep}{dep}

\title{Colored multizeta values in positive characteristic}
\author{Ryotaro Harada}
		\address{
				Tohoku University, Aza-Aoba 6-3 Aramaki, Aoba-ku, Sendai City, Miyagi 9808578, Japan.}
			\email{harada@tohoku.ac.jp}
\date{June 8, 2023}

\newtheorem{thm}{Theorem}[section]
\newtheorem{lem}[thm]{Lemma}
\newtheorem{cor}[thm]{Corollary}
\newtheorem{prop}[thm]{Proposition}
\theoremstyle{remark}
        
\theoremstyle{definition}
\newtheorem{defn}[thm]{Definition}

\newtheorem{q}[thm]{Question}

\begin{document}
\bibliographystyle{amsalpha+}

\begin{abstract}
In this paper, we study Shen-Shi's colored multizeta values in positive characteristic,
which are generalizations of multizeta values in positive characteristic by Thakur.
We establish their fundamental properties, that include their non-vanishingness,
sum-shuffle relations, $t$-motivic interpretation and linear independence.
In particular, for the linear independence results, we prove that there are no nontrivial $\overline{k}$-linear relations among the colored multizeta values of different weights.
\end{abstract}


\maketitle
\tableofcontents
\setcounter{section}{-1}

\section{Introduction}
\subsection{Background in positive characteristic}
We set $A:=\mathbb{F}_q[\theta]$ the polynomial ring in the variable $\theta$ over finite field $\mathbb{F}_q$, with quotient field $k:=\mathbb{F}_q(\theta)$, and $A_{+}$ be the set of monic polynomials in $A$.
Let $k_{\infty}$ be the completion of $k$ at the infinite place $\infty$, and $\mathbb{C}_{\infty}$ be the completion of a fixed algebraic closure $\overline{k_{\infty}}$ at $\infty$.
Further we set $k':=\overline{\mathbb{F}}_q(\theta)$ where $\overline{\mathbb{F}}_q$ is an algebraic closure of $\mathbb{F}_q$.
In 2004, Thakur (\cite{T04}) invented and studied {\it multizeta values in positive characteristic} (MZVs). He defined MZVs in the following form for each index $(s_1, \ldots, s_n)\in\mathbb{N}^n$ ($\mathbb{N}$ is the set of positive integers):
\[
	\zeta_A(s_1, \ldots, s_n):=\sum_{\substack{\deg a_1>\cdots >\deg a_n\geq 0\\ a_1, \ldots, a_n\in A_{\plus}}}\frac{1}{a_1^{s_1}\cdots a_n^{s_n}}\in k_{\infty}.
\]
For each index $\mathfrak{s}=(s_1, \ldots, s_n)\in\mathbb{N}^n$ in this paper, we note $\wt(\mathfrak{s}):=s_1+\cdots +s_n$ and $\dep(\mathfrak{s}):=n$.
Thakur showed that (i) MZVs are never zero (\cite{T09}) and (ii) MZVs satisfy certain algebraic relation, so-called the sum shuffle relation (\cite{T10}). Additionally, he and Anderson found that (iii) MZVs are related to some pre-$t$-motives via rigid analytic trivializations (\cite{AT09}). Property (i) is an analogue to the fact that the classical MZVs are nonzero. Grouping the terms by degrees, MZVs are represented with certain power sums $S_{d_i}(s_i)$ (see Proposition \ref{prop,atformula}), and Thakur obtained (i) by showing the inequality among the power sums $S_{d_i}(s_i)$. He clarified (ii) by showing that the products of power sums are written as a $\mathbb{F}_p$-linear combination of power sums. 

Anderson and Thakur found (iii) by creating a series that lifts each MZV to an element of Tate algebra, namely, the deformation series of the MZV by Anderson-Thakur polynomials. The series satisfies a relation under the Frobenius $(-1)$-twist, and then they succeeded in interpreting MZVs by constructing a suitable pre-$t$-motive and its rigid analytic trivialization. The properties (i)--(iii) are analogous to those of the classical multizeta values; then, by combining them with the Anderson-Brownawell-Papanikolas' linear-independence criterion (ABP-criterion in short), one can obtain a transcendence/linear independence result that is far beyond the classical case. For example, Chang showed a linear independence among the monomials of the MZVs and consequently proved the stronger form of direct sum conjecture in positive characteristic:
 \begin{thm}[\cite{C14}] 
There are no nontrivial $\overline{k}$-linear relations among the MZVs whose weights are different from each other.
 \end{thm}

In this paper, we aim to generalize the above results to the case of Shen-Shi's variants of MZVs, which we call {\it colored multizeta values in positive characteristic} (CMZVs). They are described as follows: for given $\mathfrak{s}=(s_1, \ldots, s_n)\in\mathbb{N}^n$ and ${\boldsymbol \xi}=(\xi_1, \ldots, \xi_n)\in(\overline{\mathbb{F}}_q^{\times})^n$ where each $\xi_i$ is the $m$-th root of unity, we define the CMZVs by
\[
	\zeta_{A}(\mathfrak{s}; \boldsymbol{\xi}):=\sum_{\substack{\deg a_1>\cdots >\deg a_n\\ a_1, \ldots, a_n\in A_+}}\frac{\xi_1^{\deg a_1}\cdots \xi_n^{\deg a_n}}{a_1^{s_1}\cdots a_n^{s_n}}\in \overline{\mathbb{F}}_q((\frac{1}{\theta})).
\]
CMZVs are specialized to the MZVs and alternating multizeta values in positive characteristic (AMZVs) by \cite{H21}, when $N=1$ and $N=q-1$, respectively. Furthermore, we can generalize the four properties mentioned above to the CMZV case.

We can express each CMZV by power sums twisted by roots of unity \eqref{power series expression}, and thus, we can show the non-vanishingness of CMZVs by Thakur's inequality for the degree of the power sums (\cite{T09}):
\begin{thm}[Theorem \ref{nonvan}]
For any ${\mathfrak s}=(s_1, \ldots, s_n)\in\mathbb{N}^n$ and ${\boldsymbol \xi}=(\xi_1, \ldots, \xi_n) \in{(\overline{\mathbb{F}}_q^{\times})}^{n}$, $\zeta_A(\mathfrak{s};\boldsymbol{\xi})$ is not zero.
\end{thm}

The power sum expression is also used to show algebraic relations for CMZVs. We can show that the product of the twisted power sums is again written by the $\mathbb{F}_p$-linear combination of the twisted power sums in Theorem \ref{thm:sumshuffle}. As a consequence, we obtain the sum-shuffle relation for CMZVs: 
\begin{thm}[Corollary \ref{cor:sumshufflecmzv}]
For ${\mathfrak a }:=(a_1, \ldots, a_r)\in\mathbb{N}^{r}$, ${\mathfrak b}:=(b_1, \ldots, b_s)\in\mathbb{N}^{s}$, ${\boldsymbol \epsilon}:=(\epsilon_1, \ldots, \epsilon_r)\in(\overline{\mathbb{F}}_q^{\times})^{r}$ and ${\boldsymbol \lambda}:=(\lambda_1, \ldots, \lambda_s)\in(\overline{\mathbb{F}}_q^{\times})^{s}$, we may express the product $\zeta_A({\mathfrak a }; {\boldsymbol \epsilon})\zeta_A({\mathfrak b} ;{\boldsymbol \lambda})$ as follows:
\begin{align}\label{shamzv}
	\zeta_{A}({\mathfrak a};{\boldsymbol \epsilon})\zeta_{A}({\mathfrak b};{\boldsymbol \lambda})=\sum_{i} f_i\zeta_A(c_{i1}, \ldots, c_{il_{i}}; \mu_{i1}, \ldots, \mu_{il_{i}})
\end{align}
for some $c_{ij}\in\mathbb{N}$, $\mu_{ij}\in\overline{\mathbb{F}}_q^{\times}$ so that $\sum_{m=1}^{r}a_m+\sum_{n=1}^{s}b_n=\sum_{h=1}^{l_i}c_{ih}$, $\prod_{m=1}^{r}\epsilon_m\prod_{n=1}^{s}\lambda_n=\prod_{h=1}^{l_i}\mu_{ih}$, $l_i\leq r+s$ and $f_i\in\mathbb{F}_p$ for each $i$. 
\end{thm}

For the $t$-motivic interpretation of CMZVs, unlike the (A) MZV case, we need to use a generalization of the pre-$t$-motive introduced in \cite{CPY11}, which we recall in Definition \ref{defn:higher level pre t-motive}. 
In \cite{CPY11}, they investigated Papanikolas theory for higher-level pre-$t$-motives together with a rigid analytic trivialization. The $t$-motivic interpretation of CMZVs is stated by the following theorem: 

\begin{thm}[Theorem \ref{rat}]
We set $\mathfrak{s}=(s_1, \ldots, s_n)\in\mathbb{N}^n$ and $\boldsymbol{\xi}=(\xi_1, \ldots, \xi_n)\in(\overline{\mathbb{F}}_q^{\times})^n$
Let $\zeta_A(\mathfrak{s}; \boldsymbol{\xi})$ be a CMZV of level $m$ so that $\xi_1, \ldots, \xi_n\in(\mathbb{F}_{q^r}^{\times})$ and $M$ be a pre-$t$-motive of level $m$, which is defined by the matrix \eqref{motive:repmatrix}. Then, $\zeta_A(\mathfrak{s}; \boldsymbol{\xi})$ appears as an entry of rigid analytic trivialization of $M$ with the specialization $t=\theta$.
\end{thm}

This theorem enable us to apply Chang's refined ABP-criterion (\cite{C09}) for CMZVs and we obtain the following theorem.
\begin{thm}[Theorem \ref{linindcmz}]
Let $w_1, \ldots, w_l\in\mathbb{N}$ be distinct. We suppose that $V_i$ is a $k'$-linearly independent subset of the set of monomials of CMZVs with total weight $w_i$ for $i=1, \ldots, l$. Then, the following union
\[
    \{ 1 \}\cup\bigcup_{i=1}^{l}V_{i}
\]
is a linearly independent set over $\overline{k}$, that is, there are no nontrivial $\overline{k}$-linear relations among the elements of $\{ 1 \}\cup\bigcup_{i=1}^{l}V_{i}$.
\end{thm}

We denote $\overline{\mathcal{CZ}}$ (resp. $\overline{\mathcal{CZ}}_w$) to be the $\overline{k}$-linear space generated by the CMZVs (resp. CMZVs of weight $w$). 
As a consequence, we prove that there are no nontrivial $\overline{k}$-linear relations among CMZVs with different weights. 
 \begin{cor}[Corollary \ref{cordecom}]
  For the distinct positive integers $w_1, w_2, \ldots$, we have the following:
$$	
	\overline{\mathcal{CZ}}=\overline{k}\oplus\bigoplus_{i\geq 1}\overline{\mathcal{CZ}}_{w_i}.
$$	
 \end{cor}
 
In the recent studies around MZVs, remarkable achievements are made by Chang-Chen-Mishiba (\cite{CCM22}) and Im-Kim-Le-Ngo Dac-Pham (\cite{IKLNDP22}). They determined the dimension and basis of the $k$-linear space generated by MZVs, namely, they solved both Zagier conjecture and Hoffman conjecture in positive characteristic. Moreover, \cite{CCM22} obtained the upper bound of the dimension in the case for $v$-adic multizeta values and \cite{IKLNDP22} determined the dimension of the $k$-linear space generated by the AMZVs. Recently, Im-Kim-Le-Ngo Dac-Pham also constructed the Hopf algebra structure for the MZVs and AMZVs in \cite{IKLNDP23a, IKLNDP23b}. One can consider the generalization of these results for AMZVs to the case of CMZVs but they are still open problems. We hope to address these issues in the near future.

For the algebraic independence of CMZVs, we do not yet fully develop it. But there is a work for CMZVs \cite{Y23} by Yeo, where he gave a $t$-motivic interpretation in a different way from our Theorem \ref{rat} and he proved the algebraic independence of AMZVs as an application combining Papanikolas' theory \cite{P08}. 

\subsection{Background in characteristic 0}

We set $N$ to be a positive integer. 
In the characteristic 0 case, the level $N$ colored multizeta values are defined as the following series for an index $(s_1, \ldots, s_n)\in\mathbb{N}^n$ of positive integers, and an index $(\xi_1, \ldots, \xi_n)\in(\mathbb{C}^{\times})^n$ of the $N$th root of unities $\mu_i\ (i=1, \ldots, n)$ satisfying $(s_1, \xi_1)\neq (1,1)$:
$$ 
	\zeta(\mathfrak{s}; \boldsymbol{\xi}):=\sum_{m_1>\cdots >m_n>0}\frac{\xi_1^{m_1}\cdots \xi_n^{m_n}}{m_1^{s_1}\cdots m_n^{s_n}}.
$$
This definition shifts to that of the multizeta values when $N=1$. 
Similar to the multizeta values, there are studies of the extended double shuffle relations (\cite{AK04}), connections to motive (\cite{D10, DG05, Gl16}). 
In studies on linear independence, Deligne and Goncharov (\cite{DG05}) obtained the upper bound for the dimension of $\mathbb{Q}$-linear space generated by the colored multizeta values of fixed weight and level. See also \cite{D10, Gl16} for more results in particular about specific level case. However, the direct sum conjecture for (colored) multizeta values (cf. \cite{Z94, G97}) is also an open problem, and, to date, there are no even partial results. 

The outline of this paper is as follows. 
In Section 1, we outline the fundamental notations and deformation series of the CMZVs. 
In Sections 2, we prove the non-vanishingness of the CMZVs and present their sum-shuffle relations.
In Section 3, we first recall the higher level pre-$t$-motive and its rigid analytic trivialization. After that, we prove Theorem \ref{rat} by constructing a pre$t$-motive so that the CMZVs are represented via rigid analytic trivialization.
In Section 4 we conclude that the positive characteristic analogue of the direct sum conjecture for CMZV is proven by using the results in the last four sections. The idea of the proof is similar to those of the (A)MZV case by \cite{C14} and \cite{H21}; thus, we give a sketch in Section 4 and leave the detailed proof in the Appendix. 

\section{Preliminaries}
\subsection{Notations}
\label{No}
We use the following notations.

\begin{itemize}
\setlength{\leftskip}{1.0cm}
\item[$\mathbb{N}=$] the set of positive integers.
\item[$q=$] a power of a prime number $p$.  
\item[$\mathbb{F}_q=$] a finite field with $q$ elements.
\item[$\theta$, $t=$] independent variables.
\item[$A=$] the polynomial ring $\mathbb{F}_q[\theta]$.
\item[$A_{+}=$] the set of monic polynomials in $A$.
\item[$A_{d+}=$] the set of elements in $A_{+}$ of degree $d$. 
\item[$k=$] the rational function field $\mathbb{F}_q(\theta)$.
\item[$k_{\infty}=$] the completion of $k$ at the infinite place $\infty$, $\mathbb{F}_q((\frac{1}{\theta}))$.
\item[$\overline{k_{\infty}}=$] a fixed algebraic closure of $k_{\infty}$.
\item[$\mathbb{C}_{\infty}=$] the completion of $\overline{k_{\infty}}$ at the infinite place $\infty$.
\item[$\overline{k}=$] a fixed algebraic closure of $k$ in $\mathbb{C}_{\infty}$.
\item[$k'=$] $\overline{\mathbb{F}}_q(\theta)$ where $\overline{\mathbb{F}}_q$ is a fixed algebraic closure of $\mathbb{F}_q$ in $\mathbb{C}_{\infty}$.
\item[$|\cdot|_{\infty}=$] a fixed absolute value for the completed field $\mathbb{C}_{\infty}$ so that $|\theta|_{\infty}=q$.
\item[$\mathbb{T}=$] the Tate algebra over $\mathbb{C}_{\infty}$, the subring of $\mathbb{C}_{\infty}\llbracket t \rrbracket$ consisting of power series convergent on the closed unit disc $|t|_{\infty}\leq 1$.
\item[$\mathbb{L}=$] the quotient field of $\mathbb{T}$.
\item[$\mathbb{E}=$] $\{\sum_{i=0}^{\infty}a_it^i\in\overline{k}\llbracket t \rrbracket\mid \lim_{i\to\infty}|a_i|_\infty^{1/i}=0,~[k_\infty(a_0,a_1,\dots):k_\infty]<\infty\}.$
\item[$\Gamma_{n+1}=$] the Carlitz gamma value, $\prod_{i}\prod^{i-1}_{j=0}(\theta^{q^i}-\theta^{q^j})^{n_i}$ where $n = \sum_{i}n_iq^i\in\mathbb{Z}_{\geq0} \ (0\leq n_i\leq q-1)$.
\end{itemize}

Qibin Shen and Shuhui Shi introduced the positive characteristic analogues of the colored multizeta values. These were first announced by Thakur in \cite{T17}.
\begin{defn}[Shen-Shi]
Let $\mathfrak{s}=(s_1, \ldots, s_n)\in\mathbb{N}^n$ and ${\boldsymbol \xi}=(\xi_1, \ldots, \xi_n)\in(\overline{\mathbb{F}}_q^{\times})^n$
where each $\xi_i$ is the $r$-th root of unity. Then, level $r$ CMZVs are defined by 
\[
	\zeta_{A}(\mathfrak{s}; \boldsymbol{\xi}):=\sum_{\substack{\deg a_1>\cdots >\deg a_n\\ a_1, \ldots, a_n\in A_+}}\frac{\xi_1^{\deg a_1}\cdots \xi_n^{\deg a_n}}{a_1^{s_1}\cdots a_n^{s_n}}\in \overline{\mathbb{F}}_q((\frac{1}{\theta})).
\]
\end{defn}
By definition, CMZVs are specialized to MZVs as a level $m=1$ case. 

In some cases, we can derive quantities as the specialization $t=\theta$ of the deformation series. For example, we fix a $(q-1)$th root of $-\theta$ and define the {\it Carlitz period} by
\[
	\tilde{\pi}:=(-\theta)^{-\frac{q}{q-1}}\prod_{i=1}^{\infty}\Bigl( 1-\frac{\theta}{\theta^{q^i}} \Bigr)\in\overline{k_{\infty}}.
\]
Then, we can interpret this by $\Omega|_{t=\theta}$, where 
\[
	\Omega:=(-\theta)^{-\frac{q}{q-1}}\prod_{i=1}^{\infty}\Bigl( 1-\frac{t}{\theta^{q^i}} \Bigr)\in\mathbb{E}.
\]  
Anderson and Thakur showed that there exists a certain two-variable polynomial, which is called {\it Anderson-Thakur polynomial} satisfying the following formula.
\begin{prop}[\cite{AT90}]\label{prop,atformula}
For $s\in\mathbb{N}$, there exists $H_{s-1}(t)\in A[t]$ such that
\begin{align}\label{Anderson-Thakur formula}
	(H_{s-1}(t)\Omega)^{(d)}|_{t=\theta}=\frac{\Gamma_s}{\tilde{\pi}}S_d(s)
\end{align}
where $S_d(s):=\sum_{a\in A_{d+}}1/a^{s}$, so-called a {\it power sum}.
\end{prop}
By the {\it twisted power sums} $S_d(s;\xi):=\xi^dS_d(s)$ and its generalization
$$S_{d}(s_1, \ldots, s_n;\xi_1, \ldots, \xi_n):=\sum_{d=d_1>d_2>\cdots>d_n\geq 0}S_{d_1}(s_1; \xi_1)\cdots S_{d_n}(s_n; \xi_n),$$ 
we can write the CMZVs by 
\begin{align}\label{power series expression}
	\zeta_{A}(\mathfrak{s}; \boldsymbol{\xi})=\sum_{d\geq 0}S_{d}(s_1, \ldots, s_n;\xi_1, \ldots, \xi_n).
\end{align}

We set the following series.  
\begin{align}\label{def:L}
L(s_1, \ldots, s_n;\xi_1, \ldots, \xi_n):=\sum_{d_1>\cdots>d_n\geq 0}\xi_1^{d_1}(H_{s_1-1}\Omega^{s_1})^{(d_1)}\cdots\xi_n^{d_n}(H_{s_n-1}\Omega^{s_n})^{(d_n)}.
\end{align}
Then, \eqref{Anderson-Thakur formula} and \eqref{power series expression} enable us to interpret CMZVs multiplied by the Carlitz period and Carlitz Gamma value via $L(s_1, \ldots, s_n;\xi_1, \ldots, \xi_n)$ at $t=\theta$ in the following form:
\begin{align}\label{deformcmzvs}
L(s_1, \ldots, s_n;\xi_1, \ldots, \xi_n)|_{t=\theta}=\frac{\Gamma_{s_1}\cdots\Gamma_{s_n}}{\tilde{\pi}^{s_1+\cdots +s_n}}\zeta_{A}(s_1, \ldots, s_n;\xi_1, \ldots, \xi_n).
\end{align}
The expressions \eqref{power series expression} and \eqref{deformcmzvs} are applied to show the sum-shuffle relations, and $t$-motivic interpretation of CMZVs. 

\section{Non-vanishing and sum shuffle relations of CMZVs}
Next, we show that the CMZVs are nonzero and satisfy the sum-shuffle relations.

\begin{thm}\label{nonvan}
For any ${\mathfrak s}=(s_1, \ldots, s_n)\in\mathbb{N}^n$ and ${\boldsymbol \xi}=(\xi_1, \ldots, \xi_n) \in{(\overline{\mathbb{F}}_q^{\times})}^{n}$, $\zeta_A(\mathfrak{s};\boldsymbol{\xi})$ is non-vanishing.
\end{thm}
\begin{proof}
We can write $\zeta_A(\mathfrak{s};\boldsymbol{\xi})$ as follows:
\begin{align}\label{power sum expression of cmzvs}
	 \zeta_A(\mathfrak{s};\boldsymbol{\xi})=&\sum_{d_1>d_2>\cdots >d_n\geq0}\xi_1^{d_1}\xi_2^{d_2}\cdots\xi_n^{d_n}S_{d_1}(s_1)S_{d_2}(s_2)\cdots S_{d_n}(s_n).
\end{align}
However, in \cite{T09}, Thakur showed that
\begin{align}\label{thakur inequality}
	\deg_{\theta}S_d(k)>\deg_{\theta}S_{d+1}(k).
\end{align}
Therefore, we have
\begin{align*}
	|\zeta_A(s_1, \ldots, s_n; \xi_1, \ldots, \xi_n)|_{\infty}&=|\sum_{d_1>d_2>\cdots >d_n\geq0}\xi_1^{d_1}\xi_2^{d_2}\cdots\xi_n^{d_n}S_{d_1}(s_1)S_{d_2}(s_2)\cdots S_{d_n}(s_n)|_{\infty}\\
	&=|S_{n-1}(s_1)S_{n-2}(s_2)\cdots S_{0}(s_n)|_{\infty}\\
	&> 0.
\end{align*}
the second equality holds by \eqref{thakur inequality}, and the last inequality follows
from $\deg_{\theta}S_0(k)=0$ and $\deg_{\theta}S_d(k)<0$ $( k>0, d>0 )$ in \cite[\S 2.2.3]{T09}. 
Thus, $\zeta_A({\mathfrak s}; {\boldsymbol \xi})$ are non-vanishing.
\end{proof}


Because of the expression \eqref{power sum expression of cmzvs}, it is sufficient to check that a product of power sums is written by the $\mathbb{F}_p$-linear combination of power sums. H.-J. Chen (\cite{Ch15}) applied partial fraction decomposition for a rational function and proved the following relation: 
\begin{prop}[\cite{Ch15} Theorem 3.1]\label{chen}
For $s_1, s_2\in\mathbb{N}$, we have
\begin{align*}
	S_d(s_1)&S_d(s_2)-S_d(s_1+s_2)=\sum_{\substack{0<j<s_1+s_2\\q-1|j}}\Delta^{j}_{s_1, s_2}S_d(s_1+s_2-j, j)
\end{align*}
where 
\begin{align*}
\Delta^{j}_{s_1, s_2}=(-1)^{s_1-1}\binom{j-1}{s_1-1}+(-1)^{s_2-1}\binom{j-1}{s_2-1}.
\end{align*}
\end{prop}
We can show that this relation holds for the twisted power sum.
\begin{prop}
For $s_1, s_2\in\mathbb{N}$ and $\xi_1, \xi_2\in\overline{\mathbb{F}}_q$, we have
\begin{align*}
	S_d(s_1;\xi_1)&S_d(s_2;\xi_2)-S_d(s_1+s_2; \xi_1\xi_2)=\sum_{\substack{0<j<s_1+s_2\\q-1|j}}\Delta^{j}_{s_1, s_2}S_d(s_1+s_2-j, j; \xi_1\xi_2, 1).
\end{align*}
\end{prop}
\begin{proof}
\begin{align*}
S_d(s_1;\xi_1)S_d(s_2;\xi_2)&=\sum_{\substack{a_1\in A_{d\plus}}}\frac{\xi_1^{\deg a_1}}{a_1^{s_1}}\sum_{\substack{a_2\in A_{d\plus}}}\frac{\xi_2^{\deg a_2}}{a_2^{s_2}}\\
							 &=\Biggl(\sum_{\substack{a_1, a_2\in A_{\plus}\\ d=\deg a_1>\deg a_2}}+\sum_{\substack{a_1, a_2\in A_{\plus}\\ d=\deg a_2>\deg a_1}}+\sum_{\substack{a_1, a_2\in A_{d\plus}\\ a_1=a_2}}+\sum_{\substack{a_1, a_2\in A_{\plus}\\ a_1\neq a_2,\ d=\deg a_1=\deg a_2}}\Biggr)\frac{\xi_1^{\deg a_1}\xi_2^{\deg a_2}}{a_1^{s_1}a_2^{s_2}}\\
							 &=S_d(s_1, s_2; \xi_1, \xi_2)+S_d(s_2, s_1; \xi_2, \xi_1)+S_d(s_1+s_2; \xi_1\xi_2)\\
							 &\qquad +\sum_{\substack{a_1, a_2\in A_{\plus}\\ a_1\neq a_2,\ \deg a_1=\deg a_2}}\frac{(\xi_1\xi_2)^{\deg a_1}}{a_1^{s_1}a_2^{s_2}}.
\end{align*}
From the partial fraction decomposition and 
\[
\sum_{f\in\mathbb{F}_q^{\times}}\frac{1}{f^j}=\begin{cases}
												-1\ &\text{if $q-1|j$},\\
												0\  &\text{otherwise},
												\end{cases}
\]
It follows that
\begin{align*}
\sum_{\substack{ a_1, a_2\in A_{+}\\ a_1\neq a_2,\ \deg a_1=\deg a_2}}\frac{1}{a_1^{s_1}a_2^{s_2}}=\sum_{\substack{0<j<s_1+s_2\\(q-1)|j}}\Biggl\{\sum_{\substack{ a_1, a_2\in A_{+}\\ \deg a_1>\deg a_2}} \frac{(-1)^{s_1-1}\binom{j-1}{s_1-1}}{a_1^{s_1+s_2-j}a_2^j  }+\sum_{ \substack{ a_1, a_2\in A_{+}\\ \deg a_2>\deg a_1} }\frac{(-1)^{s_2-1}\binom{j-1}{s_2-1}}{b^{s_1+s_2-j}a_1^{j} } \Biggr\}.
\end{align*}

By using this, we obtain
\begin{align*}
\sum_{\substack{a_1, a_2\in A_{\plus}\\ a_1\neq a_2,\ \deg a_1=\deg a_2}}\frac{(\xi_1\xi_2)^{\deg a_1}}{a_1^{s_1}a_2^{s_2}}&=\sum_{\substack{0<j<s_1+s_2\\q-1|j}}\Biggl\{\sum_{\substack{ a_1, a_2\in A_{+}\\ d=\deg a_1>\deg a_2}} \frac{(-1)^{s_1-1}\binom{j-1}{s_1-1} (\xi_1\xi_2)^d }{a_1^{s_1+s_2-j}a_2^j  }\\
&\qquad\qquad\qquad +\sum_{ \substack{ a_1, a_2\in A_{+}\\ d=\deg a_2>\deg a_1} }\frac{(-1)^{s_2-1}\binom{j-1}{s_2-1}(\xi_1\xi_2)^d}{a_2^{s_1+s_2-j}a_1^{j} } \Biggr\}\\
&=\sum_{\substack{0<j<s_1+s_2\\q-1|j}}\Delta^{j}_{s_1, s_2}S_d(s_1+s_2-j, j; \xi_1\xi_2, 1).
\end{align*}
Therefore, the desired equation holds.
\end{proof}

We can derive the higher depth case of the above relation by an induction on the depth. It is completed in the same way as the proof for Lemma 2.5 in \cite{H21}. Thus, we omit and obtain the following:  
\begin{thm}\label{thm:sumshuffle}
For ${\mathfrak a }:=(a_1, \ldots, a_r)\in\mathbb{N}^{r}$, ${\mathfrak b}:=(b_1, \ldots, b_s)\in\mathbb{N}^{s}$, ${\boldsymbol \xi}:=(\xi_1, \ldots, \xi_r)\in(\overline{\mathbb{F}}_q^{\times})^{r}$ and ${\boldsymbol \lambda}:=(\lambda_1, \ldots, \lambda_s)\in(\overline{\mathbb{F}}_q^{\times})^{s}$, we may express the product $S_d({\mathfrak a }; {\boldsymbol \xi})S_d({\mathfrak b} ;{\boldsymbol \lambda})$ as follows:
\begin{align}\label{shapowsum}
	S_{d}({\mathfrak a};{\boldsymbol \xi})S_{d}({\mathfrak b};{\boldsymbol \lambda})=\sum_{i} f_iS_d(c_{i1}, \ldots, c_{il_{i}}; \mu_{i1}, \ldots, \mu_{il_{i}})
\end{align}
for some $c_{ij}\in\mathbb{N}$, $\mu_{ij}\in\overline{\mathbb{F}}_q^{\times}$ so that $\sum_{m=1}^{r}a_m+\sum_{n=1}^{s}b_n=\sum_{h=1}^{l_i}c_{ih}$, $\prod_{m=1}^{r}\xi_m\prod_{n=1}^{s}\lambda_n=\prod_{h=1}^{l_i}\mu_{ih}$, $l_i\leq r+s$ and $f_i\in\mathbb{F}_p$ for each $i$. 
\end{thm} 
 
The coefficients are independent of $d$. Thus, by summing both sides of \eqref{shapowsum} over $d$, we obtain the following sum-shuffle relation of CMZVs where the indices satisfy the same conditions as expressed above. 
\begin{cor}\label{cor:sumshufflecmzv}
\begin{align}\label{sumshufflecmzv}
	\zeta_A({\mathfrak a};\boldsymbol{\xi})\zeta_A({\mathfrak b};{\boldsymbol \lambda})=\sum_{i} f_i\zeta_A(c_{i1}, \ldots, c_{il_{i}}; \mu_{i1}, \ldots, \mu_{il_{i}}).
\end{align}
\end{cor}
This provides an algebra structure to the linear space generated by CMZVs.

\section{$t$-motivic interpretation of CMZVs}
Next, we introduce another representation of CMZVs via certain pre-$t$-motives. Unlike the (A)MZV case, it needs the following variant, which was first introduced by Chang, Papanikolas and Yu \cite{CPY11}.

\begin{defn}\label{defn:higher level pre t-motive}\cite{CPY11}
Let $r\in\mathbb{N}$. A {\it pre-$t$-motive of level $r$} is a left $\overline{k}(t)[\sigma^r, \sigma^{-r}]$-module $M$ that is finite dimensional over $\overline{k}(t)$.
\end{defn}

The rigid analytic trivialization is also defined for this variant of the pre-$t$-motive is defined as follows:
\begin{defn}\cite{CPY11}
We set $M$ to be a level $r$ pre-$t$-motive of dimension $n$ over $\overline{k}(t)$, let $\Phi$ be a 
representation matrix of multiplication by $\sigma^{r}$ on $M$ with respect to a given basis of {\bf m
} of $M$.
The matrix $\Psi\in {\rm GL}_n(\mathbb{L})$ satisfies
\[
	\Psi^{(-r)}=\Phi\Psi
\]
is called a {\it rigid analytic trivialization} of $\Phi$.
Here, we define $\Psi^{(-r)}$ by $(\Psi^{(-r)})_{ij}:=(\Psi_{ij})^{(-r)}$.
\end{defn}

\begin{lem}
For $\mathfrak{s}=(s_1, \ldots, s_n)\in\mathbb{N}^n$, $\boldsymbol{\xi}=(\xi_1, \ldots, \xi_n)\in(\overline{\mathbb{F}}_q^{\times})^n$ and $r\in\mathbb{Z}_{\geq 0}$, 
\begin{align}\label{lem:r-twist of L}
	&L(\mathfrak{s};\boldsymbol{\xi})^{(-r)}\\
	&=(\xi_1^r\cdots\xi_n^r)^{(-r)}\sum_{i=0}^{n}\Omega^{s_{i+1}+\cdots+s_n}\Bigl(\sum_{r\geq j_n>\cdots >j_{i+1}>0}T_{s_{i+1}, j_{i+1}}(\xi_i)\cdots T_{s_n, j_n}(\xi_n)L(s_1, \ldots, s_{i}; \xi_1^{(-r)}, \ldots, \xi_{i}^{(-r)}) \Bigr)\nonumber
\end{align}
where $T_{s_i, j_i}(\xi_i)=H_{s_i-1}^{(-j_i)}\prod^{j_{i+1}}_{h_i=0}\bigl((t-\theta)^{s_i}\bigr)^{(-h_i)}(\xi_i^{-j_i})^{(-r)}$.
\end{lem}
\begin{proof}
Here we set $d_0=\infty$ and $d_{n+1}=-\infty$. By using \eqref{def:L}, 
\begin{align*}
L(\mathfrak{s};\boldsymbol{\xi})^{(-r)}&=\sum_{d_1>\cdots>d_n\geq 0}(\xi_1^{(-r)})^{d_1}(H_{s_1-1}\Omega^{s_1})^{(d_1-r)}\cdots(\xi_n^{(-r)})^{d_n}(H_{s_n-1}\Omega^{s_n})^{(d_n-r)}\\
&=\sum_{i=0}^{n}\sum_{d_1>\cdots>d_i\geq r>d_{i+1}>\cdots>d_n\geq 0}(\xi_1^{(-r)})^{d_1}(H_{s_1-1}\Omega^{s_1})^{(d_1-r)}\cdots(\xi_n^{(-r)})^{d_n}(H_{s_n-1}\Omega^{s_n})^{(d_n-r)}\\
&=\sum_{i=0}^{n}\biggl(\sum_{d_1>\cdots>d_i\geq r}(\xi_1^{(-r)})^{d_1}(H_{s_1-1}\Omega^{s_1})^{(d_1-r)}\cdots(\xi_i^{(-r)})^{d_i}(H_{s_i-1}\Omega^{s_i})^{(d_i-r)}\\
 &\quad \sum_{r>d_{i+1}>\cdots>d_n\geq 0}(\xi_{i+1}^{(-r)})^{d_{i+1}}(H_{s_{i+1}-1}\Omega^{s_{i+1}})^{(d_{i+1}-r)}\cdots(\xi_n^{(-r)})^{d_n}(H_{s_n-1}\Omega^{s_n})^{(d_n-r)}\biggr)\\
&=\sum_{i=0}^{n}\biggl((\xi_1^{(-r)})^{r}\cdots(\xi_i^{(-r)})^{r}L(s_1, \ldots, s_i; \xi_1^{(-r)}, \ldots, \xi_i^{(-r)})\\
&\quad\Omega^{s_{i+1}+\cdots +s_n}(\xi_{i+1}^{(-r)})^r\cdots(\xi_{n}^{(-r)})^{r}\sum_{r\geq j_n>\cdots >j_{i+1}>0}T_{s_{i+1}, j_{i+1}}(\xi_n)\cdots T_{s_n, j_n}(\xi_n)\biggr)
\end{align*}
Here, we set $j_{i+1}=r-d_{i+1}$. Therefore, the equation \eqref{lem:r-twist of L} holds.

The fourth equality holds by computing the $\sum_{d_1>\cdots>d_i\geq r}$ part and $\sum_{r>d_{i+1}>\cdots>d_n\geq 0}$ part as follows: 
the first part becomes
\begin{align*}
&\sum_{d_1>\cdots>d_i\geq r}(\xi_1^{(-r)})^{d_1}(H_{s_1-1}\Omega^{s_1})^{(d_1-r)}\cdots(\xi_i^{(-r)})^{d_i}(H_{s_i-1}\Omega^{s_i})^{(d_i-r)}\\
&=(\xi_1^{(-r)})^{r}\cdots(\xi_i^{(-r)})^{r}\sum_{d_1>\cdots>d_i\geq r}(\xi_1^{(-r)})^{d_1-r}(H_{s_1-1}\Omega^{s_1})^{(d_1-r)}\cdots(\xi_i^{(-r)})^{d_i-r}(H_{s_i-1}\Omega^{s_i})^{(d_i-r)} \\
&=(\xi_1^{(-r)})^{r}\cdots(\xi_i^{(-r)})^{r}L(s_1, \ldots, s_i; \xi_1^{(-r)}, \ldots, \xi_i^{(-r)}), 
\end{align*}
and the second part becomes
\begin{align*}
&\sum_{r>d_{i+1}>\cdots>d_n\geq 0}(\xi_{i+1}^{(-r)})^{d_{i+1}}(H_{s_{i+1}-1}\Omega^{s_{i+1}})^{(d_{i+1}-r)}\cdots(\xi_n^{(-r)})^{d_n}(H_{s_n-1}\Omega^{s_n})^{(d_n-r)}\\
&=\Omega^{s_{i+1}+\cdots +s_n}\sum_{r\geq j_{n}>\cdots>j_{i+1}>0}(\xi_{i+1}^{(-r)})^{r-j_{i+1}}(H_{s_{i+1}-1}\Omega^{s_{i+1}})^{(-j_{i+1})}\cdots(\xi_n^{(-r)})^{r-j_n}(H_{s_n-1}\Omega^{s_n})^{(-j_n)}\\
&=\Omega^{s_{i+1}+\cdots +s_n}\sum_{r\geq j_{n}>\cdots>j_{i+1}>0}(\xi_{i+1}^{(-r)})^{r-j_{i+1}}H_{s_{i+1}-1}^{(-j_{i+1})}\prod_{h_{i+1}=0}^{j_{i+1}}\Bigl((t-\theta)^{s_{i+1}}\Bigr)^{(-h_{i+1})}\\
&\quad \cdots(\xi_n^{(-r)})^{r-j_{n}}H_{s_n-1}^{(-j_n)}\prod_{h_n=0}^{j_{n}}\Bigl((t-\theta)^{s_{n}}\Bigr)^{(-h_n)}\\
&=\Omega^{s_{i+1}+\cdots +s_n}(\xi_{i+1}^{(-r)})^r\cdots(\xi_{n}^{(-r)})^{r}\sum_{r\geq j_{n}>\cdots>j_{i+1}>0}\Bigl\{(\xi_{i+1}^{(-r)})^{-j_{i+1}}H_{s_{i+1}-1}^{(-j_{i+1})}\prod_{h_{i+1}=0}^{j_{i+1}}\Bigl((t-\theta)^{s_{i+1}}\Bigr)^{(-h_{i+1})}\\
&\quad \cdots(\xi_n^{(-r)})^{-j_{n}}H_{s_n-1}^{(-j_n)}\prod_{h_n=0}^{j_{n}}\Bigl((t-\theta)^{s_{n}}\Bigr)^{(-h_n)}\Bigr\}\\
&= \Omega^{s_{i+1}+\cdots +s_n}(\xi_{i+1}^{(-r)})^r\cdots(\xi_{n}^{(-r)})^{r}\sum_{r\geq j_n>\cdots >j_{i+1}>0}T_{s_{i+1}, j_{i+1}}(\xi_n)\cdots T_{s_n, j_n}(\xi_n).
\end{align*}

By combining these two transformations, the fourth equality holds.
\end{proof}

\begin{thm}\label{rat}
We set $\mathfrak{s}=(s_1, \ldots, s_n)\in\mathbb{N}$ and $\boldsymbol{\xi}=(\xi_1, \ldots, \xi_n)\in(\overline{\mathbb{F}}_q^{\times})^n$.
Let $\zeta(\mathfrak{s}; \boldsymbol{\xi})$ be a CMZV of level $r$ and $M$ be a pre-$t$-motive of level $r$, which is defined by the following $\Phi$:
\begin{align}\label{motive:repmatrix}
\Phi:=
\begin{array}{rccccccll}
\ldelim({5.0}{4pt}[] 
& \prod_{i=0}^{r-1}\bigl((t-\theta)^{s_1+\cdots+s_n}\bigr)^{(-i)} &  0 & \cdots & 0 & \rdelim){5.0}{4pt}[] &\\
&\mu_1\prod_{i=0}^{r-1}\bigl((t-\theta)^{s_2+\cdots+s_n}\bigr)^{(-i)}\sum_{j_1=1}^rT_{s_1,j_1} &  \ddots & \ddots & \vdots & &\\
& \vdots  & \ddots & \prod^{r-1}_{i=0}\bigl((t-\theta)^{s_n}\bigr)^{(-i)} & 0 & &\\
&\prod_{i=1}^{n}\mu_i\sum_{r\geq j_n>\cdots >j_1\geq 1}T_{s_n, j_n}\cdots T_{s_1, j_1}  & \cdots & \mu_n\sum_{r\geq j_n\geq 1}T_{s_n, j_n} & 1 & &.\\
\end{array}
\end{align}

where each $\mu_i\in\overline{\mathbb{F}}_q\ \ (1\leq i\leq n) $ is the $(q^r-1)$th root of $\xi_i^{r}$. 
Then, $\zeta(\mathfrak{s}; \boldsymbol{\xi})$ appears as an entry rigid analytic trivialization for $\Phi$ with the specialization $t=\theta$.
\end{thm}

\begin{proof}
We need to show that the following matrix $\Psi$ is the rigid analytic trivialization of $M$, that is, $\Psi\in\GL_{n+1}(\mathbb{T})$ and $\Phi\Psi=\Psi^{(-r)}$.
\begin{align}\label{Psi}
\Psi:=
&\begin{array}{rccccccll}
\ldelim({6}{1pt}[] &\Omega^{s_1+\cdots+s_n} &  &  &  &  & &
\rdelim){6}{1pt}[]  &\\
&\mu_1L(s_1)\Omega^{s_2+\cdots+s_n} & \Omega^{s_2+\cdots+s_n} &  &
&  & & & \\
&\vdots & \mu_2L(s_2)\Omega^{s_3+\cdots+s_n} & \ddots &  &  & & & \\
&\vdots & \vdots & \ddots & \ddots & & & &\\
&\mu_1\cdots\mu_{n-1}L(s_1, \ldots, s_{n-1})\Omega^{s_n} &
\mu_2\cdots\mu_{n-1}L(s_2, \ldots. s_{n-1})\Omega^{s_n} &  &
\ddots & \Omega^{s_n} & & &\\
&\mu_1\cdots\mu_{n}L(s_1, \ldots, s_n)&\mu_2\cdots\mu_{n}L(s_2, \ldots, s_n) & \cdots & \cdots &\mu_nL(s_n) &1 & &
\end{array}\\\nonumber
&\in{\rm GL}_{n+1}(\overline{k}[[t]])
\end{align}
where 
\begin{align*}
	L(s_j, \ldots, s_i)&:=L(s_j, \ldots, s_i; \xi_j, \ldots, \xi_i)\\
				   &:=\sum_{d_j>\cdots>d_i\geq 0}\xi_j^{d_j}(H_{s_j-1}\Omega^{s_j})^{(d_j)}\cdots\xi_i^{d_i}(H_{s_i-1}\Omega^{s_i})^{(d_i)}.
\end{align*}
By Lemma \ref{psie} and \eqref{Psi gl tate}, we obtain $\Psi\in\GL_{n+1}(\mathbb{T})$.
 we prove that each $(i,j)$-th entry of $\Psi^{(-r)}$ is equal to the $(i, j)$-th entry of $\Phi\Psi$. 
\begin{align*}
(\Psi_{ij})^{(-r)}&=\biggl( \Bigl(\prod^{i-1}_{l=j}\mu_l\Bigr) L(s_j, \ldots, s_{i-1})\Omega^{s_i+\cdots+s_n}  \biggr)^{(-r)}\\
				 \intertext{by the equation \eqref{lem:r-twist of L} and $\xi_i^{(-r)}=\xi_i$,}
				  &=\Bigl(\prod^{i-1}_{l=j}\mu_l\Bigr)^{(-r)}(\xi_j^r\cdots\xi_{i-1}^r)\sum_{l=j}^{n}\Omega^{s_{l+1}+\cdots+s_{i-1}}\Bigl(\sum_{r\geq m_{i-1}>\cdots >m_l>0}T_{s_l, m_l}(\xi_l)\cdots T_{s_{i-1}, m_{i-1}}(\xi_{i-1})\\
				  &\qquad L(s_l, \ldots, s_{i-1}; \xi_l, \ldots, \xi_{i-1}) \Bigr)(\Omega^{s_i+\cdots+s_n})^{(-r)}\\
\intertext{by the definition of $\mu_l$, we have $\mu_l^{(-r)}=\xi_l^{-r}\mu_l$ and then}
&=\Bigl(\prod^{i-1}_{l=j}\mu_l\Bigr)(\Omega^{s_i+\cdots+s_n})^{(-r)}\sum_{l=j}^{n}\Omega^{s_{l+1}+\cdots+s_{i-1}}\Bigl(\sum_{r\geq m_{i-1}>\cdots >m_l>0}T_{s_l, m_l}(\xi_l)\cdots T_{s_{i-1}, m_{i-1}}(\xi_{i-1})\\
				  &\qquad L(s_l, \ldots, s_{i-1}; \xi_l, \ldots, \xi_{i-1}) \Bigr)
\end{align*}

On the other hand, by the definition of $\Phi$ and $\Psi$, 
\[
(\text{ $i$-th row of $\Phi$})=(a_{i1}, \ldots, a_{ii}, 0, \ldots, 0)
\]
where 
\[
	a_{ij}=\begin{cases}
			&\prod_{l=j}^{i-1}\mu_l\prod_{m=0}^{r-1}\bigl( (t-\theta)^{s_i+\cdots+s_n} \bigr)^{(-m)}\sum_{r\geq m_{i-1}>\cdots>m_j\geq 1}\prod^{i-1}_{g=j}T_{s_g, m_g}\quad (i>j),\\
		    &\prod^{i-1}_{l=j}\mu_l\bigl( (t-\theta)^{s_i+\cdots+s_n} \bigr)^{(-m)}\quad (i=j),
		   \end{cases}
\]
and 
\[
(\text{ $j$-th column of $\Psi$})=(0, \ldots, 0, b_{jj}, \ldots, b_{n+1j})^{\rm tr}
\]
where
\[
	b_{ij}=\Bigl(\prod^{i-1}_{l=j}\mu_l\Bigr)L(s_j, \ldots, s_{i-1})\Omega^{s_i+\cdots+s_n}.
\]
\begin{align*}
(\Phi\Psi)_{ij}&=\biggl\{\prod_{l=j}^{i-1}\mu_l\prod_{m=0}^{r-1}\bigl( (t-\theta)^{s_i+\cdots +s_n} \bigr)^{(-m)}\sum_{r\geq m_{i-1}>\cdots >m_j\geq 1}\prod_{g=j}^{i-1}T_{s_g, m_g}\biggr\}\Omega^{s_j+\cdots+s_n}\\
			   &\quad +\biggl\{\prod^{i-1}_{l=j+1}\mu_l\prod_{m=0}^{r-1}\bigl( (t-\theta)^{s_i+\cdots+s_n} \bigr)^{(-m)}\sum_{r\geq m_{i-1}>\cdots >m_{j+1}\geq 1}\prod_{g=j+1}^{i-1}T_{s_g, m_g}\biggr\}\Bigl(\prod_{l=j}^{j}\mu_l\Bigr) L(s_j)\Omega^{s_{j+1}+\cdots+s_n}\\
			   &\quad \vdots\\
			   &\quad +\prod_{l=i}^{i-1}\mu_l\prod_{m=0}^{r-1}\bigl( (t-\theta)^{s_i+\cdots+s_n} \bigr)^{(-m)}\Bigl(\prod^{i-1}_{l=j}\mu_l\Bigr)L(s_j, \ldots, s_{i-1})\Omega^{s_i+\cdots+s_n}\\
			   \intertext{by using $\prod_{m=0}^{r-1}\Bigl((t-\theta)^{s_i+\cdots +s_n}\Bigr)^{(-m)}\Omega^{s_i+\cdots+s_n}=(\Omega^{s_i+\cdots+s_n})^{(-r)}$,}  
			   &=\biggl\{\Bigl(\prod_{l=j}^{i-1}\mu_l\Bigr)(\Omega^{s_i+\cdots+s_n})^{(-r)}\sum_{r\geq m_{i-1}>\cdots >m_j\geq 1}\prod_{g=j}^{i-1}T_{s_g, m_g}\biggr\}\Omega^{s_j+\cdots+s_{i-1}}\\
			   &\quad +\biggl\{\Bigl(\prod^{i-1}_{l=j+1}\mu_l\Bigr)(\Omega^{s_{i}+\cdots+s_n})^{(-r)}\sum_{r\geq m_{i-1}>\cdots >m_{j+1}\geq 1}\prod_{g=j+1}^{i-1}T_{s_g, m_g}\biggr\}\Bigl(\prod_{l=j}^{j}\mu_l\Bigr) L(s_j)\Omega^{s_{j+1}+\cdots+s_{i-1}}\\
			   &\quad \vdots\\
			   &\quad +(\Omega^{s_{i}+\cdots+s_n})^{(-r)}\prod^{i-1}_{l=j}\mu_lL(s_j, \ldots, s_{i-1})\\
			   &=\biggl\{\Bigl(\prod_{l=j}^{i-1}\mu_l\Bigr)(\Omega^{s_i+\cdots+s_n})^{(-r)}\sum_{r\geq m_{i-1}>\cdots >m_j\geq 1}\prod_{g=j}^{i-1}T_{s_g, m_g}\biggr\}\Omega^{s_j+\cdots+s_{i-1}}\\
			   &\quad +\biggl\{\Bigl(\prod^{i-1}_{l=j}\mu_l\Bigr)(\Omega^{s_{i}+\cdots+s_n})^{(-r)}\sum_{r\geq m_{i-1}>\cdots >m_{j+1}\geq 1}\prod_{g=j+1}^{i-1}T_{s_g, m_g}\biggr\}L(s_j)\Omega^{s_{j+1}+\cdots+s_{i-1}}\\
			   &\quad \vdots\\
			   &\quad +\Bigl(\prod^{i-1}_{l=j}\mu_l\Bigr)(\Omega^{s_{i}+\cdots+s_n})^{(-r)}L(s_j, \ldots, s_{i-1})\\
			   &=\Bigl(\prod^{i-1}_{l=j}\mu_l\Bigr)(\Omega^{s_i+\cdots+s_n})^{(-r)}\sum_{l=j}^{n}\Omega^{s_{l+1}+\cdots+s_{i-1}}\Bigl(\sum_{r\geq m_{i-1}>\cdots >m_l>0}T_{s_l, m_l}(\xi_l)\cdots T_{s_{i-1}, m_{i-1}}(\xi_{i-1})\\
			   &\qquad L(s_l, \ldots, s_{i-1}; \xi_l, \ldots, \xi_{i-1}) \Bigr)
			   \end{align*}
\end{proof}

\section{Linear independence of CMZVs}
Now, we have confirmed that CMZVs do not vanish, and we are satisfied that the sum-shuffle relations are associated with the higher level pre-$t$-motives defined by \eqref{motive:repmatrix}. In this section, we show CMZVs analogue of the direct sum conjecture via those properties and the refined ABP criterion.

\begin{prop}[{\cite[Proposition 3.1.3]{ABP04}}]\label{abpprop1}
Suppose 
\[
	\Phi\in{\rm Mat}_{n}(\overline{k}[t]), \quad \psi\in {\rm Mat}_{n\times 1}(\mathbb{T})
\]
such that 
\[
	{\rm det}\Phi|_{t=0}\neq 0, \quad \psi^{(-1)}=\Phi\psi.
\]
Then,
\[
	\psi\in {\rm Mat}_{n\times 1}(\mathbb{E}).
\]
\end{prop}
This and $\Omega^{(-1)}=(t-\theta)\Omega$ immediately show that
\begin{align}\label{omegaentire}
	\Omega\in\mathbb{E}.
\end{align}

\begin{lem}\label{psie}
Let $\Psi\in{\rm GL}_{n+1}(\overline{k}[[t]])$
be as \eqref{Psi}. Then, the following holds:
\begin{align}\label{lem Psi entire}
	\Psi\in{\rm Mat}_{n+1}(\mathbb{E}).
\end{align}
\end{lem}
\begin{proof}
To apply Proposition \ref{abpprop1}, we first prove that $\Psi\in{\rm Mat}_{n+1}(\mathbb{T})$.
By \eqref{omegaentire}, $\Omega\in\mathbb{T}$. Furthermore,
$\Bigl(\prod_{n=i}^{j}\mu_n\Bigr) L(\mathfrak{s};\boldsymbol{\xi})$ belongs to $\mathbb{T}$, where $\mathfrak{s}=(s_i, \ldots, s_j)$ and $\boldsymbol{\xi}=(\xi_i, \ldots, \xi_j)$ for $1\leq i\leq j\leq r$. Because when $|t|_{\infty}\leq 1$, 
\[
\frac{ ||\mu_i\xi_i^{d_i}H_{s_i-1}^{(d_i)}\cdots\mu_j\xi_j^{d_j}H_{s_j-1}^{(d_j)}||_{\infty} }{||\bigr((t-\theta^q)\cdots(t-\theta^{q^{d_i} })\bigr)^{s_i}\cdots\bigr((t-\theta^q)\cdots(t-\theta^{q^{d_j} })\bigr)^{s_j}||_{\infty}}\rightarrow 0
\]
as $0\leq d_j<\cdots <d_i\rightarrow \infty$ ($d_i$ goes to infinity preserving $0\leq d_j<\cdots <d_i$). Here, we set $||\sum^{m}_{n= 0}a_nt^n||_{\infty}:=\max_{m\geq n \geq 0}\{ |a_n|_{\infty}\}$ for $\sum^{m}_{n=0}a_nt^n\in\overline{k}[t]$. Since the elements of $\overline{\mathbb{F}}_q$ have nothing to do with $||\cdot ||_{\infty}$, the above argument holds by changing $\epsilon_i, \gamma_i$ to $\xi_i, \mu_i$ in the proof of \cite[Lemma 4.3]{H21}.

Thus, it follows that
\begin{align}\label{psitate}
	\Psi\in{\rm Mat}_{n+1}(\mathbb{T}).
\end{align}
 
For each $1\leq i\leq n+1$, the $i$-th column $\psi_{i}\in{\rm Mat}_{n+1\times 1}(\mathbb{T})$ of matrices $\Psi$ and $\Phi\in{\rm Mat}_{n+1}(\overline{k}[t])$ in Theorem \ref{rat} satisfies $\psi_i^{(-1)}=\Phi\psi_i$. Furthermore, $\Phi$ satisfies $({\rm det}\Phi)|_{t=0}=(-\theta)^{\sum_{i=1}^n d_i}\neq 0$, where $d_i=s_i+\cdots+s_n$. 
Therefore, we can apply Proposition \ref{abpprop1} 
and then we obtain \eqref{lem Psi entire}.
\end{proof}
We have
$\Psi\in{\rm GL}_{n+1}(\mathbb{T})$
because there is a matrix 
\begin{align*}
&\Upsilon=\\\nonumber
&\begin{array}{rccccccll}
\ldelim({6}{1pt}[] &\Omega^{-(s_1+\cdots+s_n)} &  &  &  &  & &
\rdelim){6}{1pt}[]  &\\
&-\mu_1L^{\star}(s_1)\Omega^{-(s_2+\cdots+s_n)} & \Omega^{-(s_2+\cdots+s_n)} &  &
&  & & & \\
&\vdots & -\mu_2L^{\star}(s_2)\Omega^{-(s_3+\cdots+s_n)} & \ddots &  &  & & & \\
&\vdots & \vdots & \ddots & \ddots & & & &\\
&-\mu_1\cdots\mu_{n-1}L^{\star}(s_1, \ldots, s_{n-1})\Omega^{-s_n} &
-\mu_2\cdots\mu_{n-1}L^{\star}(s_2, \ldots. s_{n-1})\Omega^{-s_n} &  &
\ddots & \Omega^{-s_n} & & &\\
&-\mu_1\cdots\mu_{n}L^{\star}(s_1, \ldots, s_n)&-\mu_2\cdots\mu_{n}L^{\star}(s_2, \ldots, s_n) & \cdots & \cdots &-\mu_nL^{\star}(s_n) &1 & &
\end{array}\\\nonumber
&\in{\rm Mat}_{n+1}(\mathbb{T})
\end{align*}
where 
\begin{align*}
	L^{\star}(s_j, \ldots, s_i)&:=L^{\star}(s_j, \ldots, s_i; \xi_j, \ldots, \xi_i)\\
				   &:=\sum_{d_j\geq \cdots\geq d_i\geq 0}\xi_j^{d_j}(H_{s_j-1}\Omega^{s_j})^{(d_j)}\cdots\xi_i^{d_i}(H_{s_i-1}\Omega^{s_i})^{(d_i)}.
\end{align*}
The matrix $\Upsilon$ satisfies $\Upsilon\Psi=\Psi\Upsilon=I_{n+1}$ ($I_{n+1}$ is a size $n+1$ identity matrix) by the relations
\begin{align*}
(-1)^iL^{\star}(s_j, \ldots, s_i)&=\sum^{j}_{m=i+1}(-1)^{m-1}L(s_i, \ldots, s_j)L^{\star}(s_j, \ldots, s_m)+(-1)^jL(s_i, \ldots, s_j)\\
(-1)^jL^{\star}(s_j, \ldots, s_i)&=\sum_{m=i+1}^{j}(-1)^{m}L(s_m, \ldots, s_j)L^{\star}(s_{m-1}, \ldots, s_i)+(-1)^iL(s_i, \ldots, s_j)
\end{align*}
which are followed from the inclusion-exclusion principle for $d_j, \ldots, d_i$ similar to the proof of \cite[Lemma 4.1]{GN} and \cite[Lemma 4.2.1]{CM19}. 
Therefore 
\begin{align}\label{Psi gl tate}
\Psi\in\GL_{n+1}(\mathbb{T}).
\end{align}

Chang gave the refinement of the ABP criterion in \cite{C09} and stated it for the level $r$ pre-$t$-motive in the following form in \cite[\S 6.1]{C20}.
\begin{thm}[{\cite[Theorem 1.2]{C09}}]\label{c09thm}
Fix $r\in\mathbb{N}$. Fix a matrix $\Phi\in{\rm Mat}_{n}(\overline{k}[t])$ such that $\mathrm{det}(\Phi)$ is a polynomial in $t$ with $\mathrm{det}(\Phi)|_{t=0}\neq 0$. Fix a column vector $\psi=(\psi_1, \ldots, \psi_n)^{tr}\in{\rm Mat}_{n\times 1}(\mathbb{T})$ satisfying $\psi^{(-r)}=\Phi\psi$. Let $f\in\overline{k}^{\times}$ satisfy $f\notin\overline{\mathbb{F}}_q$ and $\mathrm{det}(\Phi|_{t=f^{q^{-ri}}})\neq 0$ for all $i\in\mathbb{N}$. Then, we have the following:
\begin{enumerate}
     \item [(i)] For every vector $\rho\in\mathrm{Mat}_{1\times n}(\overline{k})$ such that $(\rho\psi)|_{t=f}=0$, there exists a row vector $P\in\mathrm{Mat}_{1\times n}(\overline{k}[t])$ such that $P|_{t=f}=\rho$ and $P\rho=0$.  
     \item[(ii)] $\mathrm{tr.deg}_{\overline{k}(t)}\overline{k}(t)(\psi_1, \ldots, \psi_n)=\mathrm{tr.deg}_{\overline{k}}\overline{k}(\psi_1|_{t=f}, \ldots, \psi_n|_{t=f})$
\end{enumerate}
\end{thm}

The deformation series $L(\mathfrak{s})$ of CMZVs and the $t$-motivic interpretation enable us to obtain the following property. 

\begin{lem}\label{cmzproperty}
For a given CMZV $\zeta_{A}(\mathfrak{s};\boldsymbol{\xi})$, we assume that ${\rm wt}(\mathfrak{s})=w$, ${\rm dep}(\mathfrak{s})=n$ and each entry of $\boldsymbol{\xi}$ belongs to $\mathbb{F}_{q^r}^{\times}$. Then, there exists $\Phi\in{\rm Mat}_{n+1}(\overline{k}[t])$ and $\psi\in{\rm Mat}_{(n+1)\times 1}(\mathbb{E})$ with $n\geq 1$ such that:
\begin{enumerate}
     \item [(i)] $\psi^{(-r)}=\Phi\psi$ holds and $\Phi$ satisfies the condition of Theorem \ref{c09thm}$;$
     \item[(ii)] the last column of $\Phi$ is of the form $(0, \ldots, 0, 1)^{{\rm tr}};$
     \item[(iii)] for some $a\in\overline{\mathbb{F}}_q^{\times}$ and $b\in k^{\times}$, $\psi(\theta)$ is of the form with specific first and last entries
                     \[
                         \psi(\theta)=\Biggl(\frac{1}{\tilde{\pi}^w}, \ldots, a\frac{b\zeta_{A}(\mathfrak{s};\boldsymbol{\xi})}{\tilde{\pi}^w}\Biggr)^{\rm tr};
                     \]
     \item[(iv)] for some $N\in\mathbb{N}$ divisible by $r$ and some $c\in \overline{\mathbb{F}}_q^{\times}$, $\psi(\theta^{q^N})$ is of the form  
                      \[
                         \psi(\theta^{q^N})=\Biggl(0, \ldots, 0,  ac^{N}\Bigl(\frac{b\zeta_{A}(\mathfrak{s};\boldsymbol{\xi})}{\tilde{\pi}^w}\Bigr)^{q^{N}}\Biggr)^{\rm tr}.
                     \]
 \end{enumerate}
\begin{proof}
One can immediately check that $\Phi$ and $\psi_1$ such that $\Psi=(\psi_1, \ldots, \psi_{n+1})$ satisfies conditions (i)-(iii). Thus, we prove that condition (iv) holds. 
\begin{align*}
L(\mathfrak{s};\boldsymbol{\xi})&=L(\mathfrak{s})=\sum_{d_1>\cdots>d_n\geq 0}\xi_1^{d_1}\cdots\xi_n^{d_n}(H_{s_1-1}\Omega^{s_1})^{(d_1)}\cdots(H_{s_n-1}\Omega^{s_n})^{(d_n)}\\
&=\sum_{d_1>\cdots>d_n\geq 0}\frac{\xi_1^{d_1}\cdots\xi_n^{d_n}H_{s_1-1}^{(d_1)}\Omega^{s_1}\cdots H_{s_n-1}^{(d_n)}\Omega^{s_n}}{\bigl((t-\theta^q)\cdots(t-\theta^{q^{d_1}})\bigr)^{s_1}\cdots\bigl( (t-\theta^q)\cdots(t-\theta^{q^{d_n}}) \bigr)^{s_n}}
\end{align*}
We consider the decomposition $L(\mathfrak{s})=L(\mathfrak{s})_{\geq N}+L(\mathfrak{s})_{<N}$, where 
\begin{align*}
	L(\mathfrak{s})_{\geq N}&:=\sum_{d_1>\cdots>d_n\geq N}\frac{\xi_1^{d_1}\cdots\xi_n^{d_n}H_{s_1-1}^{(d_1)}\Omega^{s_1}\cdots H_{s_n-1}^{(d_n)}\Omega^{s_n}}{\bigl((t-\theta^q)\cdots(t-\theta^{q^{d_1}})\bigr)^{s_1}\cdots\bigl( (t-\theta^q)\cdots(t-\theta^{q^{d_n}}) \bigr)^{s_n}}\\
	L(\mathfrak{s})_{<N}&:=\sum_{\substack{d_1>\cdots>d_n\geq 0\\d_n<N}}\frac{\xi_1^{d_1}\cdots\xi_n^{d_n}H_{s_1-1}^{(d_1)}\Omega^{s_1}\cdots H_{s_n-1}^{(d_n)}\Omega^{s_n}}{\bigl((t-\theta^q)\cdots(t-\theta^{q^{d_1}})\bigr)^{s_1}\cdots\bigl( (t-\theta^q)\cdots(t-\theta^{q^{d_n}}) \bigr)^{s_n}}.
\end{align*}

By following the proof of \cite[Lemma 4.4]{H21}, we obtain $\bigl\{L(s_1, \ldots, s_i)\Omega^{s_{i+1}+\cdots+s_n}\bigr\}|_{t=\theta^{q^N}}=0\ (0\leq i\leq N)$ and $L(\mathfrak{s})|_{t=\theta^{q^N}}=(\xi_1\cdots\xi_n)^{N}\Bigl(\frac{ \Gamma_{s_1}\cdots\Gamma_{s_n} }{ \tilde{\pi}^{w} }\zeta_A(\mathfrak{s};\boldsymbol{\xi})\Bigr)^{q^N}$. We conclude that $\psi_1|_{t=\theta^{q^N}}$ satisfies condition (iv) with $a=\mu_1\cdots\mu_n$, $b=\Gamma_{s_1}\cdots\Gamma_{s_n}$ and $c=\xi_1\cdots\xi_n$.
\end{proof}
\end{lem}

Next, we show that the monomials of the CMZVs satisfy Lemma \ref{cmzproperty} by using the method used in the proof of \cite[Proposition 3.4.4]{C14}.
\begin{prop}\label{monocmzprop}
We let $\zeta_A(\mathfrak{s}_1;\boldsymbol{\xi}_1), \ldots, \zeta_A(\mathfrak{s}_n;\boldsymbol{\xi}_n)$ be CMZVs with weights $w_1,$ $\ldots, w_n$ and let $m_1, \ldots, m_n\in\mathbb{Z}_{\geq 0}$. Then, there exist $r\in\mathbb{N}$, matrices $\Phi\in{\rm Mat}_d(\overline{k}[t])$ and $\psi\in{\rm Mat}_{d\times 1}(\mathbb{E})$ with $d\geq 2$ such that $(\Phi, \psi, \prod_{i=1}^n\zeta_A(\mathfrak{s}_i;\boldsymbol{\xi}_i)^{m_i}, r)$ satisfies (i)-(iv) in Lemma \ref{cmzproperty}. 
\end{prop}
\begin{proof}
We take 4-tuple $(\Phi_i, \psi_i, \zeta_A({\mathfrak s}_i;{\boldsymbol \xi}_i), r_i)$, which satisfies Lemma \ref{cmzproperty} for each $i$. Then, we consider the Kronecker product $\otimes$ (see \cite[Chapter 8]{Sc}) of $\Phi_i$ and $\psi_i$ as follows:
\[
	\Phi:=\Phi_1^{\otimes m_1}\otimes\cdots\otimes\Phi_n^{\otimes m_n},\quad \psi:=\psi_1^{\otimes m_1}\otimes\cdots\otimes\psi_n^{\otimes m_n}.
\]
By our assumption, $(\Phi_i, \psi_i, \zeta_A({\mathfrak s}_i;{\boldsymbol \xi}_i), r_i)$ satisfies Lemma \ref{cmzproperty}, and thus, by using the property of the Kronecker product, which involves matrix multiplication (cf. \cite[Theorem 7.7]{Sc}) and setting $r={\rm lcm}(r_1, \ldots, r_n )$, the 4-tuple $(\Phi, \psi, \prod_{i=1}^n\zeta_A(\mathfrak{s}_i;\boldsymbol{\xi}_i)^{m_i}, r)$ does so.  
\end{proof}

\begin{defn}
Let $\zeta_A({\mathfrak s}_1;{\boldsymbol \xi}_1 ), \ldots, \zeta_A({\mathfrak s}_n;{\boldsymbol \xi}_n )$ be CMZVs of ${\rm wt}({\mathfrak s}_i)=w_i$ $(i=1, \ldots, n)$. Because $m_1, \ldots, m_n\in\mathbb{Z}_{\geq0}$ is not all zero, we define the total weight of the monomial $\zeta_A({\mathfrak s}_1;{\boldsymbol \xi}_1 )^{m_1} \cdots \zeta_A({\mathfrak s}_n;{\boldsymbol \xi}_n )^{m_n}$ as 
\[
	\sum^n_{i=1}m_iw_i.
\]
For $w\in\mathbb{N}$, we denote $CZ_{w}$ as the set of monomials of CMZVs with total weight $w$.  
\end{defn}

\begin{thm}\label{linindcmz}
Let $w_1, \ldots, w_l\in\mathbb{N}$ be distinct. We suppose that $V_i$ is a $k'$-linearly independent subset of $CZ_{w_i}$ for $i=1, \ldots, l$. Then, the following union
\[
    \{ 1 \}\cup\bigcup_{i=1}^{l}V_{i}
\]
is a linearly independent set over $\overline{k}$, that is, there are no nontrivial $\overline{k}$-linear relations among the elements of $\{ 1 \}\cup\bigcup_{i=1}^{l}V_{i}$.
\end{thm}

In the proof of this theorem, we are careful with the constant multiple by the elements in $\overline{\mathbb{F}}_q$, but the idea is basically the same as \cite{C14} and \cite{H21}. Thus, here, we give the outline of the proof and leave the details in the Appendix.

\begin{itemize}
\item[$\bullet$ Step 1] Assume on the contrary that $V_{l}$ is a $\overline{k}$-linearly dependent set.

\item[$\bullet$ Step 2] Show that $V_l$ is an $k'$-linearly dependent set by using Theorem \ref{c09thm}.

\item[$\bullet$ Step 3] Show that $V_l$ is a $k'$-linearly dependent set by Steps 1 and 2 and Theorem \ref{c09thm}.

\end{itemize} 
The above theorem provides the following corollaries. 
As a $l=1$ case of Theorem \ref{linindcmz}, we get the following:

\begin{cor}
For each $\mathfrak{s}=(s_1, \ldots, s_n)\in\mathbb{N}$ and $\boldsymbol{\xi}=(\xi_1, \ldots, \xi_n)\in(\overline{\mathbb{F}}_q^{\times})^n$, $\zeta(\mathfrak{s}; \boldsymbol{\xi})$ is transcendental over $k$.
\end{cor}

Furthermore, the theorem enables us to obtain the linear independence result among the CMZVs and lift the $k'$-linear independence of CMZVs to $\overline{k}$-linear independence. We conclude with the following consequences for $\overline{\mathcal{CZ}}$ (resp. $\overline{\mathcal{CZ}}_w$), the $\overline{k}$-linear space generated by the CMZVs (resp. CMZVs of weight $w$). 

\begin{cor}\label{cordecom}
We have the following direct sum decomposition of $\overline{\mathcal{CZ}}$:
\begin{align*}
\overline{\mathcal{CZ}}=\overline{k}\oplus\bigoplus_{w\geq 1}\overline{\mathcal{CZ}}_w.
\end{align*}
\end{cor}

\begin{cor}
We have the canonical bijective map
\begin{align*}
\overline{k}\otimes_{k'}\mathcal{CZ}\rightarrow \overline{\mathcal{CZ}}.
\end{align*}
\end{cor}

\section*{Appendix}
In this section, we state the detailed proof of Theorem \ref{linindcmz}.

\subsection{Key Lemma}
\begin{lem}\label{laslem}
Let $V_l$ be a finite $\overline{k}$-linearly dependent subset of $CZ_{w}$. Then, $V_l$ is a finite $k'$-linearly dependent subset of $CZ_{w}$. 
\end{lem}
\begin{proof}
We put $V_l=\{Z_{1}, \ldots, Z_{m}\}$. Without loss of generality, we may assume that $m\geq 2$ by Theorem \ref{nonvan} and 
\begin{align}\label{dim}
	{\rm dim}_{\overline{k}}{\rm Span}_{ \overline{k} }\{ V_l \}=m-1.
\end{align}
Again, we may assume that $Z_{1}\in{\rm Span}_{\overline{k}}\{  Z_{2}, \ldots, Z_{m}  \}$; then, by assumption \eqref{dim}, $\{ Z_{2}, \ldots, Z_{m} \}  $ is a linearly independent set over $\overline{k}$. 
By Proposition \ref{monocmzprop}, we can take the matrix $\Phi_{j}$ and the column vector $\psi_{j}$ $(1\leq j\leq m)$ so that the 4-tuple $(\Phi_{j}, \psi_{j}, Z_{j}, r_j)$ satisfies Lemma \ref{cmzproperty} (i)-(iv). We define the block diagonal matrix $\Phi$ and column vector $\psi$ as follows. 
\[
	\Phi:=\bigoplus_{j=1}^{m}\Phi_{j}\quad \text{and}\quad \psi:=\bigoplus_{j=1}^{m}\psi_{j}. 
\]
In the above, we again define the direct sum of column vectors ${\bf v}_1, \ldots, {\bf v_m}$ by $\bigoplus^{m}_{i=1}{\bf v}_i:=({\bf v}_1^{\rm tr}, \ldots, {\bf v}_m^{\rm tr})^{\rm tr}$. 
By Lemma \ref{cmzproperty}, we have 
\begin{align}\label{psiw1}
                         \psi|_{t=\theta}=\bigoplus_{j=1}^{m}\Biggl(\frac{1}{\tilde{\pi}^{w}}, \ldots, a_j\frac{b_jZ_{j}}{\tilde{\pi}^{w}}\Biggr)^{\rm tr}
\end{align}
for some $a_j\in\overline{\mathbb{F}}_q^{\times}$, $b_j\in k^{\times}$ and  
 \begin{align}\label{psinw1}
                         \psi|_{t=\theta^{q^N}}=\bigoplus_{j=1}^{m}\Biggl(0, \ldots, 0,  a_jc_j^{N}\Bigl(\frac{b_jZ_{j}}{\tilde{\pi}^{w}}\Bigr)^{q^{N}}\Biggr)^{\rm tr}
\end{align}
for some $c_j\in \mathbb{F}_q^{\times}$, $N\in\mathbb{N}$.
By using Theorem \ref{c09thm}, there exist row vectors ${\bf f}_j=({f}_{j1}, \ldots, {f}_{jd_{j}})\in{\rm Mat}_{1\times d_{j}}(\overline{k}[t])$ ($j = 1, \ldots, m$) so that if we put ${\bf F}=({\bf f}_1, \ldots, {\bf f}_{m})$, then we have
\[
	{\bf F}\psi=0, \\ \text{ ${f}_{1d_{1}}|_{t=\theta}=1$  and  ${f}_{ji}|_{t=\theta}=0$ for $1\leq i <d_{j}$}. 
\] 
Since ${f}_{1, d_{1}}|_{t=\theta}$ is the coefficient of $a_1b_1Z_{1}/(\tilde{\pi}^{w})$ in $({\bf F}\psi)|_{t=\theta}=0$ and by the assumption \eqref{dim}, $Z_{1}$ is expressed by nontrivial $\overline{k}$-linear combinations of $Z_{2}, \ldots, Z_{m}$.
We write ${\bf F}':=(1/{f}_{1d_{1}}){\bf F}$ and $d:=\sum_{j=1}^{m}d_{j}$. Note that the vector ${\bf F}'$ is of the form 
\[
	{\bf F}'=({f}'_{11}, \ldots, {f}'_{1d_{1}}, \ldots, {f}'_{m1}, \ldots, {f}'_{md_{m}} )\in{\rm Mat}_{1\times d}(\overline{k}(t))
\]
where ${f}'_{1d_{1}}=1$. We have the following from ${\bf F}\psi=0$ and ${f}_{ji}|_{t=\theta}=0$ for $1\leq i <d_{j}$: 
\begin{align}\label{qpsi0}
	{\bf F}'\psi=0\ \text{and}\ {f}'_{ji}|_{t=\theta}=0\ \text{ for all $1\leq i< d_{j}$}.
\end{align}

By using Lemma \ref{cmzproperty}, for $r={\rm lcm}(r_1, \ldots, r_m)$ we obtain ${\bf F}'^{(-r)}\Phi\psi=({\bf F}'\psi)^{(-r)}=0$ and thus
\begin{align}\label{eqnQ}
	{\bf F}'\psi-{\bf F}'^{(-r)}\Phi\psi=({\bf F}'-{\bf F}'^{(-r)}\Phi)\psi=0.
\end{align}
The last column of the matrix $\Phi_{j}$ is $(0, \ldots, 0, 1)^{\rm tr}$ for each $j$, and consequently, the $d_{1}$-th entry of row vector ${\bf F'}-{\bf F'}^{(-r)}\Phi$ is zero since the $d_{1}$-th entry of the row vectors ${\bf F'}$ and ${\bf F'}^{(-r)}\Phi$ are 1. 
The $\sum_{i=1}^{j}d_{i}$-th column of $\Phi$ is 
\begin{align*}
\begin{array}{rcll}
\ldelim( {7}{4pt}[] & 0 & \rdelim) {7}{4pt}[] & \rdelim\}{4}{10pt}[$\sum_{i=1}^{j}d_{i}$] \\
& \vdots & & \\
& 0 & & \\
& 1 & &  \\
& 0 & & \\
& \vdots & & \\
&0 & &.
\end{array}
\end{align*}
Then, the $\sum_{i=1}^{j}d_{i}$-th entry of the row vector ${\bf F}'^{(-r)}\Phi$ is ${f}'^{(-r)}_{jd_{j}}$. Thus, the $\sum_{i=1}^{j}d_{i}$-th entry of ${\bf F}'-{\bf F}'^{(-r)}\Phi$ is written as follows:
\[ 
	{f}'_{jd_{j}}-{f}'^{(-r)}_{jd_{j}}\quad \text{ for $j=1, \ldots, m$.} 
\]
We claim that ${f}'_{jd_{j}}-{f}'^{(-r)}_{jd_{j}}=0$ for $j=2, \ldots, m$. 
Indeed, if there exists some $2\leq j\leq m$ such that ${f}'_{jd_{j}}-{f}'^{(-r)}_{jd_{j}}\neq 0$, we can derive the contradiction in the following way: 

Let us take a sufficiently large $N\in\mathbb{N}$ so that $({f}'_{jd_{j}}-{f}'^{(-r)}_{jd_{j}})|_{t=\theta^{q^N}}\neq 0$ and all entries of $({\bf F}'-{\bf F'}^{(-r)}\Phi)$ are regular at $t=\theta^{q^{N}} $. By using \eqref{psinw1} and substituting $t=\theta^{q^N}$ in \eqref{eqnQ}, we obtain
\begin{align*}
	\Bigl\{({\bf F}'-{\bf F}'^{(-r)}\Phi)\psi\Bigl\}|_{t=\theta^{q^N}}&=\Bigl\{({\bf F}'-{\bf F}'^{(-r)}\Phi)\Bigr\}|_{t=\theta^{q^N}}\bigoplus^{m}_{j=1}\Biggl(0, \ldots, 0, a_jc_j^{N}\Bigl(\frac{b_jZ_{j}}{\tilde{\pi}^{w}}\Bigr)^{q^N}\Biggr)^{\rm tr}\\
	&=\sum_{j=1}^{m}({f}'_{jd_{j}}-{f}'^{(-r)}_{jd_{j}})|_{t=\theta^{q^N}}a_jc_j^{N}\Bigl(\frac{b_jZ_{j}}{\tilde{\pi}^{w}}\Bigr)^{q^N}\\
	&=\frac{1}{\tilde{\pi}^{w}}\sum_{j=1}^{m}({f}'_{jd_{j}}-{f}'^{(-r)}_{jd_{j}})|_{t=\theta^{q^N}}a_jc_j^{N}\Bigl(b_jZ_{j}\Bigr)^{q^N}=0.
\end{align*}
Thus, combined with ${f}'_{1d_{1}}-{f}'^{(-r)}_{1d_{1}}=1-1=0$, we obtain the nontrivial $\overline{k}$-linear relations among $Z_{2}^{q^N}, \ldots, Z_{m}^{q^N}$ as follows: 
 \[
 	\sum_{j=2}^{m}({f}'_{jd_{j}}-{f}'^{(-r)}_{jd_{j}})|_{t=\theta^{q^N}}a_jc_j^{N}\Bigl(b_jZ_{j}\Bigr)^{q^N}=0.
 \]
 Then, by taking the $q^{N}$th root of the relation, we obtain the following nontrivial $\overline{k}$-linear relation among $Z_{2}, \ldots, Z_{m}$:
 \[
 	\sum_{j=2}^{m}\Bigl\{({f}'_{jd_{j}}-{f}'^{(-r)}_{jd_{j}})|_{t=\theta^{q^N}}a_jc_j^{N}\Bigr\}^{\frac{1}{q^N}}b_jZ_{j}=0.
 \]
 This contradicts our assumption that $\{ Z_{2}, \ldots, Z_{m} \}$ is a $\overline{k}$-linearly independent set. 

Therefore, we obtain ${f}'_{jd_{j}}-{f}'^{(-r)}_{jd_{j}}=0$ for $j=2, \ldots, m$, and this equation shows the following:
 \begin{align}\label{coefrat}
 {f}'_{jd_{j}}\in \mathbb{F}_{q^r}(t) \quad (j=2, \ldots, m).
 \end{align} 
 By substituting $t=\theta$ in equation ${\bf F'}\psi=0$, equations \eqref{psiw1} and \eqref{qpsi0} enable us to obtain the following equalities:
 \begin{align*}
 	({\bf F'}\psi)|_{t=\theta}&=\bigl(0, \ldots, 0, {f}'_{1d_{1}}(\theta), \ldots, 0, \ldots, 0, {f}'_{md_{m}}(\theta)\bigr)\bigoplus_{j=1}^{m}\Biggl(\frac{1}{\tilde{\pi}^{w}}, \ldots, a_j\frac{b_jZ_{j}}{\tilde{\pi}^{w}}\Biggr)^{\rm tr}\\
 					      &=\sum_{j=1}^{m}({f}'_{jd_{j}}|_{t=\theta})a_j\frac{b_jZ_{j}}{\tilde{\pi}^{w}}=\frac{1}{\tilde{\pi}^{w}}\sum_{j=1}^{m}({f}'_{jd_{j}}|_{t=\theta})a_jb_jZ_{j}=0.
 \end{align*}
By ${f}'_{1d_{1}}=1$ and \eqref{coefrat}, we have the following nontrivial $k$-linear relation among $Z_{1}, \ldots, Z_{m}$:
\[
	\sum_{j=1}^{m}({f}'_{jd_{j}}|_{t=\theta})a_jb_jZ_{j}=0.
\]
Therefore, we obtain the desired claim. 
\end{proof}

\subsection{Proof of Theorem \ref{linindcmz}}
\begin{proof}
We may assume that $w_l>\cdots>w_1$ without loss of generality. By definition,
$CZ_{w_i}$ is a finite set for each $i=1, \ldots, l$, and thus, its subset $V_{i}$ is also finite. Let each $V_{i}$ consist of $\{ Z_{i1}, \ldots, Z_{im_i}  \}$ where $Z_{ij}\in CZ_{w_i}\ (j=1, \ldots, m_i)$ are the same total weight $w_i$. 
The proof is by induction on $l$.

We require, on the contrary, that $\{1\}\cup\bigcup_{i=1}^lV_i$ is an $\overline{k}$-linearly dependent set 
and then proceed to our proof by assuming the existence of nontrivial $\overline{k}$-linear relations involving $V_l$. 

For $1\leq i\leq l$ and $1\leq j\leq m_l$, by combining Theorem \ref{rat} with Proposition \ref{monocmzprop}, the matrices 
\begin{align}\label{mats}
    \Phi_{ij}\in{\rm Mat}_{d_{ij}}(\overline{k}[t])\quad {\rm and}\quad \psi_{ij}\in{\rm Mat}_{d_{ij}\times 1}({\mathbb E})
\end{align}
so that $d_{ij}\geq 2$ and each $(\Phi_{ij}, \psi_{ij}, Z_{ij}, r_{ij})$ satisfies Lemma \ref{cmzproperty}.

For matrix $\Phi_{ij}$ and column vector $\psi_{ij}$, we define the following block diagonal matrix and column vector:
\[
    \tilde{\Phi}:=\bigoplus_{i=1}^{l}\biggl( \bigoplus^{m_i}_{j=1}(t-\theta)^{w_l-w_i}\Phi_{ij} \biggr)\quad \text{and}\quad \tilde{\psi}:=\bigoplus^{l}_{i=1}\biggl( \bigoplus^{m_i}_{j=1}\Omega^{w_l-w_i}\psi_{ij} \biggr).
\]
In this proof, we define the direct sums of any matrices $M_1, \ldots, M_m$ and any column vectors ${\bf v}_1, \ldots, {\bf v_m}$ whose entries belong to $\mathbb{C}_{\infty}((t))$ by 

\begin{align*}
\bigoplus^{m}_{i=1}M_i:=
\begin{array}{rccccll}
\ldelim({4.0}{4pt}[] 
& M_1 &  &  &  &\rdelim){4.0}{4pt}[] &\\
&  & M_2 &  &  & &\\
&  &  & \ddots &  &  &\\
&  &  &  &  M_m & &\\
\end{array}
\end{align*}
and
\[
\bigoplus^{m}_{i=1}{\bf v}_i:=({\bf v}_1^{\rm tr}, \ldots, {\bf v}_m^{\rm tr})^{\rm tr}
\]
respectively. Here, ``tr'' stands for the transpose of a vector.

By the requirement, $\{ 1 \}\cup\bigcup_{i=1}^l V_i$ is a linearly dependent set over $\overline{k}$. Thus, there exists a nonzero vector
\[
	\rho=({\bf v}_{11}, \ldots, {\bf v}_{1m_1}, \ldots, {\bf v}_{l1}, \ldots, {\bf v}_{lm_l}) 
\] 
 such that 
\begin{align*}
 \rho\cdot(\tilde{\psi}|_{t=\theta})&=\rho\cdot\bigoplus_{i=1}^{l}\bigoplus^{m_i}_{j=1}\Biggl(\frac{1}{\tilde{\pi}^{w_l}}, \ldots, a_{ij}\frac{b_{ij}Z_{ij}}{\tilde{\pi}^{w_l}}\Biggr)^{\rm tr}\\
 &=\frac{1}{\tilde{\pi}^{w_l}}({\bf v}_{11}, \ldots, {\bf v}_{1m_1}, \ldots, {\bf v}_{l1}, \ldots, {\bf v}_{lm_l}) \bigoplus^{l}_{i=1}\bigoplus^{m_i}_{j=1}\Biggl(1, \ldots, a_{ij}b_{ij}Z_{ij}\Biggr)^{\rm tr}=0, 
\end{align*}
 where ${\bf v}_{ij}\in {\rm Mat}_{1\times d_{ij}}(\overline{k})$ for $1\leq i\leq l$ and $1\leq j\leq m_i$. Then, we have the following nontrivial $\overline{k}$-linear relation: 
 \[
 	({\bf v}_{11}, \ldots, {\bf v}_{1m_1}, \ldots, {\bf v}_{l1}, \ldots, {\bf v}_{lm_l}) \bigoplus^{l}_{i=1}\bigoplus^{m_i}_{j=1}\Biggl(1, \ldots, a_{ij}b_{ij}Z_{ij}\Biggr)^{\rm tr}=0.
 \]
 In the beginning of this proof, we assumed that there exist nontrivial $\overline{k}$-linear relations between $V_l$ and $\{1 \}\cup\bigcup_{i=1}^{l-1}V_i$, and then for some $1\leq s\leq m_l$, the last entry of ${\bf v}_{ls}$ is nonzero. The last entry in ${\bf v}_{li}$ is a coefficient of $a_{li}b_{li}Z_{li}$ for $1\leq i\leq m_l$ in the above relation. By using Theorem \ref{c09thm}, we have ${\bf F}:=({\bf f}_{11}, \ldots, {\bf f}_{1m_1}, \ldots, {\bf f}_{l1}, \ldots, {\bf f}_{lm_l})$ where ${\bf f}_{ij}=(f_{i1}, \ldots, f_{id_{ij}})\in{\rm Mat}_{1\times d_{ij}}(\overline{k}[t])$ for $1\leq i\leq l$, $1\leq j\leq m_i$ and satisfies
 \[
 	{\bf F}\tilde{\psi}=0\quad \text{and}\quad {\bf F}|_{t=\theta}=\rho. 
 \] 
 The last entry of ${\bf f}_{ls}$ is a nontrivial polynomial because the last entry of ${\bf v}_{ls}$ is not zero. We choose a sufficiently large $N\in\mathbb{Z}$ so that ${\bf f}_{ls}|_{t=\theta^{q^N}}\neq 0$ and $\boldsymbol{\xi}^{q^N}=\boldsymbol{\xi}$ for all factors $\zeta_A(\mathfrak{s};\boldsymbol{\xi})$ appear in all $Z_{ij}$ $(1\leq i\leq l, 1\leq j\leq m_i)$. 
We rewrite equation $({\bf F}\tilde{\psi})|_{t=\theta^{q^N}}=0$ by using $\Omega|_{t=\theta^{q^N}}=0$, Lemma \ref{cmzproperty} (iv) and the definition of $\tilde{\psi}$ as follows:
\begin{align*}
	&({\bf F}\tilde{\psi})|_{t=\theta^{q^N}}\\
	&=({\bf f}_{11}, \ldots, {\bf f}_{1m_1}, \ldots, {\bf f}_{l1}, \ldots, {\bf f}_{lm_l})|_{t=\theta^{q^N}}\bigoplus^{l}_{i=1}\bigoplus^{m_i}_{j=1}\Omega^{w_l-w_i}|_{t=\theta^{q^N}}\Biggl(0, \ldots, 0, a_{ij}c_{ij}^{N}\Bigl(\frac{b_{ij}Z_{ij}}{\tilde{\pi}^{w_i}}\Bigr)^{q^N}\Biggr)^{\rm tr}\\
	\intertext{when $i\neq l$, $\Omega^{w_l-w_i}|_{t=\theta^{q^N}}=0$ and thus}
	&=({\bf f}_{l1}, {\bf f}_{l2}, \ldots, {\bf f}_{lm_l})|_{t=\theta^{q^N}}\bigoplus^{m_l}_{j=1}\Biggl(0, \ldots, 0, a_{lj}c_{lj}^{N}\Bigl(\frac{b_{lj}Z_{lj}}{\tilde{\pi}^{w_l}}\Bigr)^{q^N}\Biggr)^{\rm tr}\\
	&=\sum_{j=1}^{m_l}(f_{ld_{lj}}|_{t=\theta^{q^N}})a_{lj}c_{lj}^{N}\Bigl(\frac{b_{lj}Z_{lj}}{\tilde{\pi}^{w_l}}\Bigr)^{q^N}=0.
\end{align*}
Thus, we obtain the following nontrivial $\overline{k}$-linear relation with some $f_{ld_{ls}}\neq0$:
\[
	\sum_{j=1}^{m_l}(f_{ld_{lj}}|_{t=\theta^{q^N}})a_{lj}c_{lj}^{N}\Bigl(b_{lj}Z_{lj}\Bigr)^{q^N}=0.
\] 
 Therefore, by taking the $q^N$th root of the above $\overline{k}$-linear relation, we obtain a nontrivial relation for $\{  Z_{l1}, \ldots, Z_{lm_l}  \}$ as follows. 
\[
	\sum_{j=1}^{m_l}\Bigl\{(f_{ld_{lj}}|_{t=\theta^{q^N}})a_{lj}c_{lj}^{N}\Bigr\}^{\frac{1}{q^N}}b_{lj}Z_{lj}=0.
\] 
This shows that $V_l$ is a $\overline{k}$-linearly dependent set.  
Then, by using Lemma \ref{laslem}, we show that $V_l$ is a $k'$-linearly dependent subset.
However, it contradicts the condition that $V_l$ is the $k'$-linearly independent set. Therefore, we finish the proof. 
\end{proof}

\section*{Acknowledgements}

The author gratefully acknowledges Professor Dinesh Thakur for informing him the definition of CMZVs and Professor Chieh-Yu Chang for providing him fruitful comments about the project. The author also thanks the JSPS Research Fellowships, JSPS Overseas Research Fellowships and the National Center for Theoretical Sciences in Hsinchu for their support during the research project. This work is also supported by JSPS KAKENHI Grant Number JP22J00006.

\end{document}